\documentclass[11pt]{amsart}   	
\usepackage{geometry}             
\usepackage[dvipsnames]{xcolor}   		             		
\usepackage{graphicx}				
\usepackage{amssymb,mathrsfs}
\usepackage{amsthm,amsmath,stmaryrd}
\usepackage{tikz}
\usepackage{tikz-cd}
\usepackage{accents,upgreek,enumerate}
\usepackage[headings]{fullpage}
\usepackage{bm}
\usepackage{mathtools}
\usepackage[all]{xy}
\usepackage{caption}
\usepackage[euler-digits]{eulervm}
\usepackage[normalem]{ulem}

\tikzset{
  commutative diagrams/.cd, 
  arrow style=tikz, 
  diagrams={>=stealth}
}

\newtheorem{innercustomthm}{Theorem}
\newenvironment{customthm}[1]
  {\renewcommand\theinnercustomthm{#1}\innercustomthm}
  {\endinnercustomthm}

  \makeatletter
\def\@tocline#1#2#3#4#5#6#7{\relax
  \ifnum #1>\c@tocdepth % then omit
  \else
    \par \addpenalty\@secpenalty\addvspace{#2}%
    \begingroup \hyphenpenalty\@M
    \@ifempty{#4}{%
      \@tempdima\csname r@tocindent\number#1\endcsname\relax
    }{%
      \@tempdima#4\relax
    }%
    \parindent\z@ \leftskip#3\relax \advance\leftskip\@tempdima\relax
    \rightskip\@pnumwidth plus4em \parfillskip-\@pnumwidth
    #5\leavevmode\hskip-\@tempdima
      \ifcase #1
       \or\or \hskip 1em \or \hskip 2em \else \hskip 3em \fi%
      #6\nobreak\relax
    \dotfill\hbox to\@pnumwidth{\@tocpagenum{#7}}\par
    \nobreak
    \endgroup
  \fi}
\makeatother

\usetikzlibrary{calc}
\usetikzlibrary{fadings}
\usetikzlibrary{decorations.pathmorphing}
\usetikzlibrary{decorations.pathreplacing}

\newcounter{marginnote}
\DeclareMathAlphabet{\mathpzc}{OT1}{pzc}{m}{it}

\usepackage[backref=page]{hyperref}
\hypersetup{
  colorlinks   = true,          %Colors links instead of ugly boxes
  urlcolor     = violet,          %Color for external hyperlinks
  linkcolor    = blue,          %Color of internal links
  citecolor   = violet             %Color of citations
}

\newtheorem{theorem}{Theorem}[subsection]
\newtheorem{corollary}[theorem]{Corollary}
\newtheorem{lemma}[theorem]{Lemma}
\newtheorem{proposition}[theorem]{Proposition}

\newtheorem{quasi-theorem}[theorem]{Quasi-Theorem}

\theoremstyle{definition}
\newtheorem{definition}[theorem]{Definition}

\newtheorem{remark}[theorem]{Remark}
\newtheorem{questions}{Question}

\newtheorem{example}[theorem]{Example}
\newtheorem{blank remark}[theorem]{}

\newtheorem{not1}[theorem]{Notation}

\newcommand{\PP}{\mathbb{P}}         
\newcommand{\QQ} {{\mathbb Q}}		
\newcommand{\RR} {{\mathbb R}}		
\newcommand{\ZZ} {{\mathbb Z}}		

\DeclareMathOperator{\ch}{ch}
\DeclareMathOperator{\Td}{Td}

\newcommand{\cal}{\mathcal}

\def\cA{{\cal A}}

\def\cK{{\mathsf K}}

\def\cM{{\cal M}}

\def\fM{\mathfrak{M}}

\newcommand{\plC}{\scalebox{0.8}[1.3]{$\sqsubset$}}

\newcommand{\Mbar}{\overline{\cM}\vphantom{\cM}}

\def\ev{{\mathrm{ev}}}

\makeatletter
\def\blfootnote{\xdef\@thefnmark{}\@footnotetext}
\makeatother

\title{Gromov--Witten theory and invariants of matroids}

\date{}

\author{Dhruv Ranganathan {\it \&} Jeremy Usatine}

\address{Dhruv Ranganathan \\ Department of Pure Mathematics {\it \&} Mathematical Statistics\\
University of Cambridge, Cambridge, UK}
\email{\href{mailto:dr508@cam.ac.uk}{dr508@cam.ac.uk}}

\address{Jeremy Usatine \\ Department of Mathematics\\
Brown University, Providence, USA}
\email{\href{mailto:jeremy_usatine@brown.edu}{jeremy\_usatine@brown.edu}}

\begin{document}

\begin{abstract}
We use techniques from Gromov--Witten theory to construct new invariants of matroids taking value in the Chow groups of spaces of rational curves in the permutohedral toric variety. When the matroid is realizable by a complex hyperplane arrangement, our invariants coincide with virtual fundamental classes used to define the logarithmic Gromov--Witten theory of wonderful models of arrangement complements, for any logarithmic structure supported on the wonderful boundary. When the boundary is empty, this implies that the quantum cohomology ring of a hyperplane arrangement's wonderful model is a combinatorial invariant, i.e., it depends only on the matroid. When the boundary divisor is maximal, we use toric intersection theory to convert the virtual fundamental class into a balanced weighted fan in a vector space, having the expected dimension. We explain how the associated Gromov--Witten theory is completely encoded by intersections with this weighted fan. We include a number of questions whose positive answers would lead to a well-defined Gromov--Witten theory of non-realizable matroids.
\end{abstract}
\maketitle

%\eject
%\setcounter{tocdepth}{1}
%\tableofcontents

\section{Introduction}

The purpose of this article is to use the logarithmic Gromov--Witten theory of wonderful compactifications to construct new invariants of matroids. Our main results are the construction of a canonical quantum deformation of the Chow ring of a realizable matroid and a tropical correspondence theorem for virtual counts of logarithmic rational curves on arrangement complements. 

\subsection{Wonderful geometry of the permutohedron} The theory of matroids is intimately related to the toric variety associated to the permutohedron. A geometric development of the basic theory of matroids and this relationship can be found in the article of Katz~\cite{Kat16}. 

A loop-free matroid $M$ of rank $d+1$ on $\{0, \dots, n\}$ determines a $d$-dimensional unimodular fan $\Delta_M$ in the vector space $\RR^{n+1}/\RR(1, \dots, 1)$, called the projective Bergman fan\footnote{The Bergman fans in this paper will be taken with respect to the maximal building set, see~\cite{dCP95,FY04}.} of $M$. We set $\Delta_n = \Delta_{B_{n+1}}$, where the Boolean matroid $B_{n+1}$ is the uniform rank $n+1$ matroid on $\{0, \dots, n\}$. Concretely, the fan $\Delta_n$ is the normal fan to the permutohedron,  the first barycentric subdivision of the fan of $\PP^n$, and $\Delta_M$ is a subfan of $\Delta_n$.

Suppose $\cA$ is a hyperplane arrangement in $\mathbb{P}^d$ realizing $M$. By linear algebra, there is an essentially unique inclusion $\mathbb{P}^d \hookrightarrow \mathbb{P}^n$ such that $\cA$ consists of the restriction to $\mathbb{P}^d$ of the coordinate hyperplanes. A famous construction of de Concini--Procesi constructs a modification of $\PP^d$
$$\iota: W_\cA \hookrightarrow X(\Delta_n),$$
as follows~\cite{dCP95}.  The variety $X(\Delta_n)$ is the toric variety associated to the fan $\Delta_n$, and explicitly, can be recognized as the blowup of $\mathbb P^n$ along coordinate strata in increasing order of dimension. The variety $W_\cA$ is the strict transform of $\mathbb P^d$ and is called the wonderful model. The key property of the construction is that the restriction of the toric boundary $\partial X(\Delta_n)$ is a simple normal crossings divisor on $W_\cA$. We write $\partial W_\cA$ for this divisor. 

Let $D_X\subset \partial X(\Delta_n)$ be any union of toric boundary divisors and let $D_{ W}\subset\partial W_\cA$ be the intersection with $ W_{\cA}$. Logarithmic Gromov--Witten theory probes the geometry of a simple normal crossings pairs by maps from curves with tangency along $D_X$. The central object for us is the moduli space of genus $0$ logarithmic stable maps $\mathsf K_\Gamma(W_{\mathcal A})$ and the corresponding space $\mathsf K_\Gamma(X(\Delta_n))$. The symbol $\Gamma$ encodes the \textit{numerical data} of the Gromov--Witten problem, which selects (i) the divisor $D_X$, (ii) the curve class, (iii) the tangency data, see Section~\ref{sec: discrete-data}.

\subsection{Main results} In the main body, we construct a Chow homology class
\[
[M]^{{\mathrm{vir}}}_\Gamma \ \ \textrm{in } A_\star\mathsf K_\Gamma(X(\Delta_n)).
\]
associated to any, possibly non-realizable, matroid and a choice of numerical data $\Gamma$. Our first result is that when the matroid $M$ is realizable over the complex numbers, this class coincides with a familiar class from Gromov--Witten theory. We preserve the notation above for the statement of the theorem. 

\begin{customthm}{A}
\label{VirtualFundamentalClassOfMatroidIsVirtualFundamentalClassOfArrangement}
Let $M$ be a rank $d+1$ matroid on $\{0, \dots, n\}$, let $\cA$ be a complex hyperplane arrangement in $\mathbb{P}^d$ realizing $M$, and let $W_\cA$ be the wonderful model with respect to the maximal building set of $\cA$. Let $\iota: W_\cA \hookrightarrow X(\Delta_n)$ is the inclusion of the wonderful model into the permutohedral toric variety, and
\[
\cK_\Gamma(\iota): \cK_\Gamma(W_\cA) \hookrightarrow \cK_\Gamma(X(\Delta_n))
\]
the induced morphism of stable mapping spaces, obtained by composing a stable map with the inclusion $\iota$. Then
\[
	\cK_\Gamma(\iota)_\star[\cK_\Gamma(W_\cA)]^{\mathrm{vir}} = [M]^{{\mathrm{vir}}}_\Gamma
\]
for any choice of numerical data $\Gamma$.
\end{customthm}

It is useful to place this result in the context of other household invariants of arrangement complements. For example, the fundamental group of the complement of an arrangement is \textit{not} a combinatorial invariant~\cite{Ryb11}. The combinatorial invariance is also non-obvious from the point of view of Gromov--Witten theory which is \textit{deformation} {invariant} by construction. The space of arrangements realizing a given matroid, its {realization space}, can be highly disconnected, and therefore, the result goes significantly beyond the deformation invariance of Gromov--Witten theory.  

The class $[M]^{{\mathrm{vir}}}_\Gamma$ can be taken as the definition of the logarithmic Gromov--Witten virtual fundamental class when $M$ does not have a realization, after pushforward to the space of maps to $X(\Delta_n)$. However, it is not immediate that the Gromov--Witten invariants of $M$ are well defined, see Section~\ref{sec: questions}. 

A straightforward consequence of the result is that the quantum cohomology of the wonderful model of a complex arrangement is a matroid invariant, giving a quantum deformation of the result of Feichtner and Yuzvinsky~\cite{FY04}. In the statement of the next theorem, fix a polarization of $X(\Delta_n)$ and give $W_{\cA_1}$ and $W_{\cA_2}$ the polarizations induced by their inclusions into $X(\Delta_n)$. We consider the quantum cohomology with respect to these polarizations.

\begin{customthm}{B}
\label{QuantumCohomologyIsCombinatorialInvariant}
Let $\cA_1, \cA_2$ be hyperplane arrangements in $\mathbb{P}^d$, and let $W_{\cA_1}, W_{\cA_2}$ be their wonderful models with respect to the maximal building set, respectively. If $\cA_1$ and $\cA_2$ have the same underlying matroid, then
\[
	QH^\star(W_{\cA_1}) \cong QH^\star(W_{\cA_2}).
\]
\end{customthm}

When the boundary $D_X$ is the toric boundary, the theory becomes combinatorial. Fixing this logarithmic structure, there is an associated tropical space $\mathsf T_\Gamma(\Delta_n)$ parameterizing tropical rational curves with a balanced map to $\Delta_n$. Consider the subcomplex $\mathsf T_\Gamma(\Delta_M)$ of curves in $\Delta_M$. Note that this is typically not pure dimensional. In Section~\ref{sec: virtual-mw}, we construct a balanced weighted subfan $c_\Gamma(M)$ of $\mathsf T_\Gamma(\Delta_n)$ of the expected dimension. We show that $c_\Gamma(M)$ completely controls the Gromov--Witten theory of $M$. 

\begin{customthm}{C}[Tropical virtual fundamental class]
If $M$ admits a realization by a complex hyperplane arrangement with wonderful model $W$, then the subcomplex $c_\Gamma(M)$ is a union of cones of $\mathsf T_\Gamma(\Delta_M)$ of the expected dimension. Moreover, the logarithmic Gromov--Witten theory of $(W,\partial W)$ with numerical data $\Gamma$ can be uniquely reconstructed from the weighted balanced fan $c_\Gamma(M)$. 
\end{customthm}

The reconstruction procedure is explained in Section~\ref{sec: reconstruction}. First examples of the object $c_\Gamma(M)$ are given in Section~\ref{sec: examples}. The result generalizes the correspondence theorem for rational curves in toric varieties, proved at the level of moduli spaces in~\cite{CMR14b,R15b}. 

The virtual fundamental classes of matroids have properties that are reminiscent of familiar matroid invariants. Given a subdivision of the matroid polytope of $M$, the virtual fundamental classes of the matroid $M$ and the initial matroids of the subdivision are related by the degeneration formula~\cite{R19}. Similarly, the virtual fundamental classes of a direct sum of matroids can be related to those of the factors~\cite{R19b}. We expect these will be important properties in the further study of matroid Gromov--Witten theory. It is important to note, however, that the virtual fundamental classes are neither valuative nor satisfy a product rule for direct sums in the traditional sense. 

When $\partial W$ is the full wonderful boundary, the virtual fundamental class has an additional pleasant property. We have constructed matroid virtual fundamental classes using the wonderful compactification with respect to the maximal building set, in the sense of~\cite{dCP95}, however, when $\partial W$ is the wonderful boundary, the birational invariance of the logarithmic Gromov--Witten classes implies that the complex $c_\Gamma(M)$ depends only on the support of $\Delta_M$ in $\Delta_n$, rather than on the additional data of its fan structure~\cite{AW}. 

\subsection{Previous work} Our work is inspired by interactions between matroid theory and algebraic geometry. The Chow ring of a matroid was introduced by Feichtner and Yuzvinsky, as the Chow ring of a certain non-complete toric variety, building on work of de Concini and Procesi~\cite{dCP95,FY04}. The matroid Chow ring has many of the properties of the cohomology ring of a smooth projective algebraic variety, as was shown in work of Adiprasito, Huh, and Katz~\cite{AHK18}. Our results show that in the realizable case, the cohomology ring admits a canonical quantum deformation, and we believe that such a deformation should always exist, see Section~\ref{sec: questions}. 

The purely combinatorial matroid virtual fundamental class $c_\Gamma(M)$ should be compared to tropical correspondence theorems proved in~\cite{NS06} for toric varieties. If the Bergman fan $\Delta_M$ is a vector space, the result reduces to~\cite{Gro14,Gro15,NS06,R15b}. We note that an independent and interesting approach to constructing polyhedral complexes that resemble virtual fundamental classes is pursued by Gathmann, Markwig, and Ochse in~\cite{GMO17,GO17}. It would be interesting to examine how their approach relates to ours. 

The idea of constructing a Gromov--Witten theory purely in the world of matroids has spiritual parallels. For example, the Chern--Schwarz--Macpherson class of a matroid is a balanced tropical cycle that recovers the CSM classes of complex arrangement complements~\cite{dMRS}. Speyer constructed an invariant of a matroid by associating to it the $K$-theory class of the structure sheaf of an associated torus orbit closure in a Grassmannian~\cite{Spe09}; its properties were studied further by Fink and Speyer~\cite{FS12}.  The Chow and $K$-theory of matroids are connected by a beautiful recent paper of Berget, Eur, Spink, and Tseng~\cite{BEST}, who connect wonderful models of matroids to equivariant vector bundles on toric varieties. In fact, this construction plays a crucial role in our study. In related work, Dastidar and Ross have defined matroid $\psi$--classes, treating the Chow ring of a matroid as analogous to the moduli space of pointed rational curves~\cite{DR21}. 

When $M$ is complex realizable, the points of the tropical cycle $c_\Gamma(M)$ parameterize maps from tropical rational curves in $\Delta_M$. In the special case where the degree of these rational curves is $1$, one obtains spaces of tropical lines in tropical linear spaces, which can be interpreted as a tropical Fano variety, or as a tropical flag manifold~\cite{BEZ,JMRS,Lam18}. Our results may therefore be viewed as a non-linear partial generalization of these. 

Finally, our results fundamentally concern integrals on the space of stable maps to the permutohedral toric variety, which is also the subject of the papers~\cite{BK05,GKP14,KR16}. 

\subsection*{Acknowledgements} We are grateful to Dan Abramovich, Christopher Eur, Dagan Karp, and Navid Nabijou for helpful conversations related to this work and to Hannah Markwig for comments on an earlier draft. D.R. was an academic visitor at the Chennai Mathematical Institute in 2018 when this project began, and is grateful to the institute for tea, biscuits, and other hospitality. The paper benefited from the careful reading and helpful comments of an anonymous referee. 

\noindent
D.R. was supported by EPSRC New Investigator Grant EP/V051830/1.

%Our final two theorems are concerned with the logarithmic Gromov--Witten theory of the pair $(W_{\cA},\partial W_{\cA})$, that is, the boundary divisor is the preimage of the hyperplane arrangement. Both results have a similar flavour, namely that the virtual fundamental classes of matroids behave well with respect to natural matroidal properties. 

%\begin{customthm}{D}[Product formula]
%
%\end{customthm}
%
%
%
%\begin{question}
%\label{AreGromovWittenInvariantsIndependentOfPrecursor}
%Let $M$ be a matroid on $\{0, \dots, n\}$, and let $\beta_1, \dots, \beta_m \in A^\star(X(\Delta_n))$ be such that $\beta_i|_{X(\Delta_M)} = 0$ for some $i \in \{1, \dots, m\}$. Do we have
%\[
%	\left(\ev_1^\star(\beta_1) \cup \dots \cup \ev_m^\star(\beta_m)\right) \cap [M]^{\mathrm{vir}}_\Gamma = 0
%\]
%for all choices of numerical data $\Gamma$?
%\end{question}
%

\section{Virtual fundamental classes of wonderful models}

We work over an algebraically closed field of characteristic $0$ and construct the virtual classes of matroids. The section involves three ingredients: the tautological vector bundle cutting out the wonderful model of a matroid~\cite{BEST}, the Grothendieck--Riemann--Roch theorem~\cite{CG07}, and the functoriality of the virtual fundamental class~\cite{KKP}. They will be combined to obtain a combinatorial formula for the virtual fundamental class of a wonderful model for a complex arrangement complement.

\subsection{Conventions on logarithms}
We prefer to work here over the category of fine logarithmic schemes rather than the more common convention of fine and saturated logarithmic schemes used in~\cite{AC11,Che10,GS13}. The geometric reason for this is that natural forgetful map
\[
\mathsf K_\Gamma(X)\to \Mbar_{0,r}(X,\beta)
\]
to the moduli space of ordinary stable maps is an embedding in the fine case, but is finite in the fine and saturated case. For statements that require a global embedding into a smooth object, the former is more convenient. The literature does not use the saturation hypothesis in a significant way, and it is easily excised without affecting the Gromov--Witten theory. As this is a technical discussion and likely well-known to experts, we delay this to Section~\ref{sec: fine-v-fine-saturated}. 

\subsection{Conventions on discrete data}\label{sec: discrete-data} Let $X$ be a smooth projective variety and $D$ a simple normal crossings divisor with components $D_1,\ldots, D_s$. Logarithmic Gromov--Witten theory probes the geometry of a simple normal crossings pair $(X,D)$ by the intersection theory on moduli spaces $\mathsf K(X,D)$ of logarithmic stable maps
\[
F: (C,p_1,\ldots, p_r)\to (X,D).
\]
The domain is allowed to be any nodal pointed curve, and the morphism satisfies a number of properties in addition to stability, which the reader may find in~\cite{GS13}. We provide a working definition for the reader. The interior of the moduli space parameterizes maps such that $F^{-1}(D)$ is a collection of marked points. On this interior, the curve $C$ is smooth and the scheme theoretic order of tangency of $F$ at each marked point $p_i$ with each boundary component $D_j$ is a fixed nonnegative integer $c_{ij}$. This contact order is locally constant in flat families of logarithmic stable maps. The space $\cK(X,D)$ is a compact moduli space that contains this locus as an open substack. The new objects in the moduli problem are logarithmic maps from nodal curves subject to a stability condition. The fundamental property of the compact moduli space is that the contact orders $c_{ij}$ are well-defined at \textit{every point} in the moduli space $\mathsf K(X,D)$, and moreover, the numbers $c_{ij}$ are still locally constant in flat families. The genus and number of marked points $r$ are also locally constant. 

It is standard to fix the following discrete data when studying logarithmic Gromov--Witten invariants: (i) the genus of the domain curve $C$, (ii) the number $r$ of marked points on the curve, (iii) the orders $c_{ij}$ of contact of each marked point $p_i$ with the boundary divisor $D_j$. 

We \underline{depart} from standard conventions and let $\Gamma$ package the data of the choice of boundary divisor $D$, the genus $g$, which will always be $0$ for us, the number $r$ of marked points, and their contact orders. Let $\mathsf K_\Gamma(X)$ be the associated moduli space of logarithmic stable maps. It is proved in~\cite{AC11,Che10,GS13} that this is a proper Deligne--Mumford stack of expected dimension
\[
\mathsf{vdim} = (\dim X-3)(1-g)+\int_\beta c(T_X^{\mathsf{log}})+r
\]
The moduli space carries a canonical Chow homology class in this expected dimension called the \textit{virtual fundamental class}. Gromov--Witten invariants are defined as integrals of tautologically defined cohomology classes on the moduli space against the virtual fundamental class. In addition to the original papers, a gentle introduction to the subject may be found in the introductory sections of~\cite{R15b}.

\subsection{The moduli space} Suppose $X$ is a smooth projective variety with a fixed simple normal crossings boundary divisor $D_X$. We consider a globally generated vector bundle $\mathbb E$ with a section $s$ such that the vanishing locus $W = \mathbb V(s)$ has the following properties. 
\begin{enumerate}[(i)]
\item The subscheme $W\subset X$ is smooth of codimension equal to the rank of of $\mathbb E$.
\item The intersection $W\cap D_X$ is a simple normal crossings divisor.
\item For each irreducible component $D_i$ of $D_X$, if the intersection $W\cap D_i$ is nonempty then it is smooth.
\end{enumerate}

\begin{remark}
In our context, $X$ will be the toric variety associated to the fan $\Delta_n$, which is the normal fan of the $n$-dimensional permutohedron. The divisor $D_X\subset X$ will be a subset of the toric boundary, which will later be fed into the discrete data $\Gamma$. The set $W$ will be the wonderful compactification of a hyperplane arrangement complement, as constructed previously. The construction of the vector bundle will be explained shortly. 
\end{remark}

In this situation the divisor $W\cap D_X$ is a simple normal crossings divisor, and the natural map $W\hookrightarrow X$ is strict in the sense of logarithmic geometry. There is an induced an inclusion of moduli spaces:
\[
\mathsf K_\Gamma(W)\hookrightarrow \mathsf K_\Gamma(X).
\]
The vector bundle $\mathbb E$ induces a sheaf on the moduli space by a universal push/pull construction. Consider the universal diagram
\[
\begin{tikzcd}
\mathscr C\arrow{d}{\pi}\arrow{r}{F} & X \\
\mathsf K_\Gamma(X). 
\end{tikzcd}
\]
Define $\mathbb F$ to be the pushforward sheaf $\pi_\star F^\star \mathbb E$. In this context, the moduli space of maps to $X$ can be described using the sheaf $\mathbb F$. 

\begin{theorem}
\label{theoremBundleOnModuliSpace}
The sheaf $\mathbb F$ is a vector bundle and possesses a section $s_{\mathbb F}$ whose stack theoretic vanishing locus is equal to the substack
\[
\mathsf K_\Gamma(W)\hookrightarrow \mathsf K_\Gamma(X).
\]
The Chern character of $\mathbb F$ is given by
\[
	\ch(\mathbb F) = \pi_\star(\Td(\pi) \cdot F^\star \ch(\mathbb E)).
\]
\end{theorem}

\begin{proof}
Let $C$ be a rational curve, and consider a morphism
\[
F: C\to X.
\]
We claim that the cohomology group $H^1(C,f^\star \mathbb E)$ vanishes. Indeed, by construction the vector bundle $\mathbb E$ is a quotient of the trivial bundle; since the curve $C$ has arithmetic genus $0$, the claim is immediate. We have the universal diagram
\[
\begin{tikzcd}
\mathscr C\arrow{d}{\pi}\arrow{r}{F} & X \\
\mathsf K_\Gamma(X). 
\end{tikzcd}
\]
Consider the pullback $F^\star\mathbb E$ of the vector bundle $\mathbb E$ to the universal curve. The higher cohomology groups of this vector bundle on the fibers of $\pi$ have been shown to vanish. By the theorem on cohomology and base change, the pushforward sheaf $R^0\pi_\star F^\star \mathbb E$ is a vector bundle $\mathbb F$.  The section $s_{\mathbb E}$ of $\mathbb E$ therefore gives rise to a section $s_{\mathbb F}$ of the vector bundle $\mathbb F$. The scheme theoretic vanishing locus of $s_{\mathbb F}$ tautologically describes the subfunctor of $\mathsf K_\Gamma(X)$ where the universal curve is scheme theoretically contained in the vanishing locus of the section $s_{\mathbb E}$, namely in $W$. The first result follows. 

The second result follows from a standard application of Grothendieck--Riemann--Roch, which will relate the Chern characters of the sheaves $F^\star \mathbb E$ on the universal curve, and of $\mathbb F$. We briefly explain how to make the application. Since we work in the category of fine logarithmic schemes, by a result of Wise, there is an embedding
\[
\mathsf K_\Gamma(X)\hookrightarrow \overline{\mathcal M}_{0,r}(X,\beta),
\]
into the space of ordinary stable maps, obtained by forgetting the logarithmic data~\cite[Corollary~1.2]{Wis16b}. Moreover, the universal curve of the source is pulled back from the target of this map. The Grothendieck--Riemann--Roch theorem may be applied to the universal curve, exactly as explained in~\cite[Appendix~1]{CG07}, and leads immediately to the claimed result. 
\end{proof}

\subsection{Logarithmic quantum Lefschetz}\label{sectionLogQuantumLefschetz} The moduli space of logarithmic stable maps $\mathsf K_\Gamma(W)$ carries a natural virtual fundamental class~\cite{GS13}. We examine this class after pushforward to the ambient space $\mathsf K_\Gamma(X)$. The description of the moduli space above as the zero locus of a bundle section also produces a class, namely the virtual Poincar\'e dual to the Euler class of the bundle $\mathbb F$. 

\begin{theorem}\label{thm: virtual-compatibility}
The following equality holds in the Chow homology of $\mathsf K_\Gamma(X)$:
\[
[\mathsf K_\Gamma(W)]^{\mathsf{vir}} = c_{\mathsf{top}}(\mathbb F)\cap [\mathsf K_\Gamma(X)]^{\mathsf{vir}}.
\]
\end{theorem}

We recall how the virtual fundamental class on $\mathsf K_\Gamma(W)$ is constructed, following the point of view given by Abramovich and Wise~\cite{AW}. Let $\mathscr A_X$ denote the Artin fan of $X$ with the logarithmic structure dictated by $\Gamma$. Explicitly, this is an open substack of the quotient $[X/T]$ of $X$ by its dense torus, consisting of torus orbits that are contained in the logarithmic boundary of $X$. The stack $\mathscr A_X$ is a simple normal crossings pair in the smooth topology.   

\begin{lemma}
The stack $\mathfrak M(\mathscr A_X)$ of logarithmic maps from curves to $\mathscr A_X$ is logarithmically smooth. Fixing the genus to be $g$ and the contact orders at all marked points by the datum $\Gamma$, the stack $\mathfrak M_{g,\Gamma}(\mathscr A_X)$ of logarithmic maps with these discrete data is irreducible of dimension $3g-3+n$, where $n$ is the number of marked points.
\end{lemma}

We work in the genus $0$ case and drop the genus from the notation. The morphism $W\hookrightarrow X$ is strict, and it follows that the Artin fan $\mathscr A_W$ is an open substack of $\mathscr A_X$. Concretely, it is the complement of those points of $\mathscr A_X$ corresponding to torus orbits that are disjoint from $W$. In particular
\[
\mathscr A_W\hookrightarrow \mathscr A_X
\]
is a strict open immersion. The corresponding inclusion $\fM_\Gamma(\mathscr A_W)\hookrightarrow \fM_\Gamma(\mathscr A_X)$ is also an open immersion. For the purposes of defining the virtual class below, we may therefore ignore $\mathscr A_W$ altogether and replace it with $\mathscr A_X$. 

The composite morphism $W\to X\to\mathscr A_X$ has relative tangent bundle equal to the logarithmic tangent bundle of $W$. The composition also gives rise to a map of moduli spaces
\[
\mu: \mathsf K_\Gamma(W)\to \mathfrak M_\Gamma(\mathscr A_X). 
\]
The vertical tangent space whose fiber at a moduli point $[C,f]$ is given by $H^0(C,f^\star T_W^{\mathsf{log}})$. The vertical obstruction space is given by $H^1(C,f^\star T_W^{\mathsf{log}})$. By the Riemann--Roch theorem, the difference between these dimensions is everywhere constant. A globalized version of this fact is that the forgetful morphism above is equipped with a relative perfect obstruction theory, see~\cite[Section~6]{AW}, noting as we have above that $\fM_\Gamma(\mathscr A_W)\hookrightarrow \fM_\Gamma(\mathscr A_X)$ is an open immersion. In particular, there is a virtual pullback morphism on the Chow group:
\[
\mu^!: A_\star(\mathfrak M_\Gamma(\mathscr A_X);\QQ)\to A_\star(\mathsf K_\Gamma(W);\QQ). 
\]
The virtual pullback of the fundamental class gives rise to the virtual fundamental class. If $W$ is replaced with $X$ above, the map $\mu$ is smooth; the pullback map coincides with smooth pullback on Chow groups. 

\subsubsection{Proof of Theorem~\ref{thm: virtual-compatibility}} 

The argument is well-known in the non-logarithmic context, and we need only manoeuvre ourselves into a situation where known results on virtual structures imply the result~\cite{KKP}. By hypothesis the inlcusion $W\hookrightarrow X$ is a strict morphism of pairs: divisorial components of $\partial W$ are precisely the nonempty intersections of divisors of $X$ with $W$. It therefore follows that the relative cotangent complex of this morphism coincides with the relative logarithmic cotangent complex. In particular, the kernel of the morphism of logarithmic cotangent bundles is the usual conormal bundle $W$ in $X$:
\[
0\to N_{W/X}^\vee\to \Omega_X^{\mathsf{log}}|_W\to \Omega_W^{\mathsf{log}}\to 0.
\]
Since $W$ is cut out of $X$ by a section of $\mathbb E$, we have an identification of $N_{W/X}$ with the vector bundle $\mathbb E|_W$.  By pulling back this sequence to the universal curve over $\mathsf K_\Gamma(W)$, taking its derived pushforward, and rotating, we obtain a distinguished triangle in the derived category of coherent sheaves on $\mathsf K_\Gamma(W)$. By unraveling the definition of the obstruction theory given above, we obtain a sequence of morphisms to the relative cotangent complexes of our moduli spaces as described below. We let $j:\mathsf K_\Gamma(W)\hookrightarrow \mathsf K_\Gamma(X)$ denote the inclusion.
\[
\begin{tikzcd}
j^\star (R\pi_\star F^\star T^{\mathsf{log}}_X)^\vee\arrow{r}\arrow{d}& (R\pi_\star F^\star T_W^{\mathsf{log}})^\vee\arrow{r}\arrow{d}&j^\star \mathbb F[1]\arrow{r}\arrow{d}& j^\star (R\pi_\star F^\star T^{\mathsf{log}}_X)^\vee[1]\arrow{d}\\
j^\star L_{\mathsf K_\Gamma(X)/\mathfrak M_\Gamma(\mathscr A_X)} \arrow{r} & L_{\mathsf K_\Gamma(W)/\mathfrak M_\Gamma(\mathscr A_X)} \arrow{r} &j^\star L_{\mathsf K_\Gamma(W)/\mathsf K_\Gamma(X))} \arrow{r}&j^\star L_{\mathsf K_\Gamma(X)/\mathfrak M_\Gamma(\mathscr A_X)}[1].
\end{tikzcd}
\]
It follows that the two different obstruction theories on $\cK_\Gamma(W)$ -- one coming from the standard theory and the other from the vector bundle $\mathbb F$, coincide; this statement now follows from the result of Kim--Kresch--Pantev~\cite{KKP}; see also~\cite[Remark~5.17]{Mano12}. 
\qed

\begin{remark}\label{remarkVirtualCompatibilityAndModifications}
The following variation on the above argument will be useful. We have a sequence of morphisms
\[
\cK_\Gamma(W)\hookrightarrow \cK_\Gamma(X)\to \mathfrak M_\Gamma(\mathscr A_X). 
\]
The logarithmic algebraic stack $\mathfrak M_\Gamma(\mathscr A_X)$ is equipped with a tropicalization~\cite[Section~2]{R19}. Any subdivision of this tropicalization determines logarithmic modifications of all three of these spaces. The logarithmic modifications of $\cK_\Gamma(W)$ and $\cK_\Gamma(X)$ are both equipped with virtual fundamental classes by~\cite[Section~3]{R19}. The argument above ensures that the virtual fundamental classes of the modifications are still related by the top Chern class of the vector bundle $\mathbb F$ above. 
\end{remark}

\subsection{Specializing to the wonderful geometry} We now specialize to let $X$ be the toric variety associated to the permutohedral fan $\Delta_n$ and let $W$ be the wonderful model of a hyperplane arrangement complement. We may exhibit $W$ as a zero of a section of a vector bundle on $X$. The construction was discovered by Berget--Eur--Spink--Tseng~\cite{BEST} and is closely related to~\cite{Kap93,Tev07}. The hyperplane arrangement underlying $W$ canonically determines a point in a Grassmannian $\mathbb G(d,n)$. The $\mathbb G_m^n$-orbit closure of this point determines a toric variety $Y$ which inherits the tautological quotient bundle from $\mathbb G(d,n)$. There is a canonical resolution $X\to Y$, and the pullback of the tautological quotient bundle to $X$ will be denoted $\mathbb E$. 

Let $M$ be the matroid associated to the hyperplane arrangement above. Note that it is a matroid on a set of $n+1$ elements. The rank of the matroid is $d+1$. We will assume for convenience throughout that $M$ is loop free. The divisor $D_X$ induces a divisor $D_W$ on $W$ by intersection. The divisor $D_X$ and the compatible divisor $D_W$ will be implicit in $\Gamma$ below. 

\begin{theorem}
\label{theoremBESTbundleSectionChern}
Let $W$ and $X$ be as above. The vector bundle $\mathbb E$ on $X$ is globally generated and possesses a transverse section $s_{\mathbb E}$ whose scheme theoretic vanishing locus is
\[
W\hookrightarrow X.
\]
The total Chern class $c(\mathbb E)$ is equal to the tautological quotient total Chern class $c(M)$ of the matroid $M$. 
\end{theorem}

The tautological quotient Chern classes of a matroid can be found in~\cite[Definition~3.9]{BEST}.

\begin{proof}
The bundle $\mathbb E$ is the tautological quotient bundle as defined in \cite[Definition 1.2]{BEST}, which as a quotient of a trivial bundle, is globally generated. The existence of the section $s_{\mathbb E}$ is \cite[Theorem 7.10]{BEST}, and the statement about $c(\mathbb E)$ is an immediate consequence of \cite[Proposition 3.7]{BEST}.
\end{proof}

\subsection{Virtual fundamental classes of matroids}\label{subsectionVFCOfMatroid}

Theorems \ref{theoremBundleOnModuliSpace}, \ref{thm: virtual-compatibility}, and \ref{theoremBESTbundleSectionChern} together tell us that the image of $[\mathsf K_\Gamma(W)]^{\mathsf{vir}}$ in the Chow homology of $\mathsf K_\Gamma(X)$ can be written in terms of the tautological quotient Chern class $c(M)$. We now make this explicit.

For $i \in \mathbb Z_{>0}$, let $p_i \in \mathbb Q[x_1, x_2, \dots]$ be recursively defined by
\[
	p_i = x_1 p_{i-1} - x_2 p_{i-2} + ... + (-1)^i x_{i-1} p_1 + (-1)^{i+1} i x_i,
\]
and for all $j \in \mathbb Z_{\geq 0}$, let $\kappa_{\Gamma, j}(M)$ be the degree $j$ term of
\[
	\pi_\star\left(\mathsf{Td}(\pi) \cdot \left( n-d +  \sum_{i \geq 1} p_i\Big( (F^\star(c_\ell(M)))_{\ell \in \mathbb Z_{>0}} \Big)  \right) \right) \in A^\star(K_\Gamma(X)).
\]
Set $e_\Gamma(M) \in A^\star(\mathsf K_\Gamma(X))$ to be
\[
	 e_\Gamma(M) = q_{\deg(c_1(M) \cap \beta) + n - d}\Big( (\ell! \kappa_{\Gamma,\ell}(M))_{\ell \in \mathbb Z_{>0}} \Big),
\]
where $\beta$ is the curve class encoded by $\Gamma$, and the $q_i \in \mathbb Q[x_1, x_2, \dots]$ are recursively defined for $i \in \mathbb Z_{>0}$ by
\[
	(-1)^{i+1} i q_i =  x_i - x_{i-1} q_1 + x_{i-2} q_2 - ... + (-1)^{i-1}
x_1 q_{i-1}.
\]

\begin{definition}\label{definitionVFCForMatroid}
The \emph{virtual fundamental class} of $M$ with respect to $\Gamma$ is the class
\[
	[M]_\Gamma^{\mathsf{vir}} = e_\Gamma(M) \cap [\mathsf K_\Gamma(X)]^{\mathsf{vir}} \in A_\star( \mathsf K_\Gamma(X)).
\]
\end{definition}

\begin{remark}
Note that the definition $[M]_\Gamma^{\mathsf{vir}}$ depends only the discrete data $\Gamma$ and the matroid $M$, and in particular, makes sense even when the matroid $M$ is not realizable. 
\end{remark}

We now see that Theorem \ref{VirtualFundamentalClassOfMatroidIsVirtualFundamentalClassOfArrangement} is an immediate consequence of Definition \ref{definitionVFCForMatroid} and Theorems \ref{theoremBundleOnModuliSpace}, \ref{thm: virtual-compatibility}, and \ref{theoremBESTbundleSectionChern}.

\begin{proof}[Proof of Theorem \ref{VirtualFundamentalClassOfMatroidIsVirtualFundamentalClassOfArrangement}]

By the Girard--Newton formula, the $p_i$ and $q_i$ are the universal polynomials describing the Chern character in terms of the Chern classes and vice-versa, see e.g., \cite[Example 3.2.3]{Ful98}. Thus for all $j \in \mathbb Z_{\geq 0}$, Theorems \ref{theoremBESTbundleSectionChern} and \ref{theoremBundleOnModuliSpace} imply that $\kappa_{\Gamma, j}(M)$ is the degree $j$ term of $\ch(\mathbb F)$, so
\[
	c_{\mathsf{top}}(\mathbb F) = q_{\deg(c_1(M) \cap \beta) + n - d}\Big( (\ell! \kappa_{\Gamma,\ell}(M))_{\ell \in \mathbb Z_{>0}} \Big) = e_\Gamma(M),
\]
where we note that $\mathrm{rk}(\mathbb F) = \deg(c_1(M) \cap \beta) + n - d$ by Riemann--Roch for vector bundles on curves. Thus Theorem \ref{thm: virtual-compatibility} gives the desired result.

\end{proof}

\subsection{Fine moduli spaces of maps}\label{sec: fine-v-fine-saturated} We record here how to adapt the results in the literature to work in the fine logarithmic setting, rather than the fine and saturated one. In this section, we set $X$ be a simple normal crossings pair with divisor $D$. We require four statements: (i) the moduli space of fine logarithmic maps to $X$ is a proper Deligne--Mumford stack, (ii) the moduli space is equipped with a virtual fundamental class, (iii) this virtual fundamental class is equal to the pushforward of the virtual fundamental class on the space of fine and saturated maps, and (iv) the equality of virtual classes is compatible with evaluation morphisms. 

The moduli space of fine logarithmic stable maps to $X$ is certainly a fine logarithmic algebraic stack, by the main results of~\cite{Wis16a,Wis16b}. There is a diagram comparing the fine moduli problem with the fine and saturated one:
\[
\begin{tikzcd}
\mathsf K_\Gamma(X)^{\mathsf{fs}}\arrow{r}\arrow{d} & \mathsf K_\Gamma(X)^{\mathsf{fine}}\arrow{d} \\
\fM_\Gamma(\mathscr A_X)^{\mathsf{fs}}\arrow{r}&\fM_\Gamma(\mathscr A_X)^{\mathsf{fine}}
\end{tikzcd}
\]
where $\mathsf A_X$ denotes the Artin fan of $X$. Both vertical maps are strict by the argument in~\cite[Lemma~4.1]{AW}, and therefore the square is Cartesian in the category of algebraic stacks. 

The lower horizontal map
\[
\begin{tikzcd}
\fM_\Gamma(\mathscr A_X)^{\mathsf{fs}}\arrow{r}&\fM_\Gamma(\mathscr A_X)^{\mathsf{fine}}
\end{tikzcd}
\]
is obtained from the saturation map, see for instance~\cite{Wis16b}. However, since the source and target are logarithmically smooth in, respectively, the fine and saturated and fine categories, the map is simply the normalization. It is therefore both proper and birational, and identifies fundamental classes under pushforward. Since the vertical maps carry the same perfect obstruction theory, it now follows from the Costello--Herr--Wise comparison theorem that the fine and saturated virtual fundamental class pushes forward to the virtual fundamental class on the space of fine maps~\cite{Cos06,HW21}. 

Finally, the evaluation maps from $\mathsf K_\Gamma(X)^{\mathsf{fs}}$ to the strata of $X$ and the forgetful map to the moduli stack of prestable curves factor through the space $\Mbar_{0,r}(X,\beta)$ of ordinary stable maps, and then by~\cite[Theorem~1.1]{Wis16b}, they necessarily factor through the space $\mathsf K_\Gamma(X)^{\mathsf{fine}}$ as well. By the projection formula, for the purposes of computing logarithmic Gromov--Witten invariants, i.e. integrals against the virtual fundamental class of cohomology classes pulled back from evaluations and from the stack of curves, we may use the theory of fine logarithmic maps. 

\section{Quantum cohomology of wonderful models}

Keeping the notation from the previous section, we will assume that in the choice of discrete data $\Gamma$, the divisor $D_X$ is empty, and we will show how in this special case Theorem \ref{VirtualFundamentalClassOfMatroidIsVirtualFundamentalClassOfArrangement} implies that the quantum cohomology of the wonderful model is a combinatorial invariant.

We first set some notation for the relevant curve classes. Let $U = X(\Delta_M) \subset X$, and let $A_1(M)$ denote the subgroup of $A_1(X)$ consisting of all $\beta$ that satisfy $\gamma \cap \beta = 0$ for all $\gamma \in \ker( A^1(X) \to A^1(U))$. Also let $\iota: W \hookrightarrow X$ denote the inclusion.

\begin{proposition}\label{PropositionImageOfCurveClassesIsCombinatorial}
The pushfoward map $A_1(W) \to A_1(X)$ is injective and has image $A_1(M)$.
\end{proposition}

\begin{proof}
The pullback map $\iota^\star: A^1(X) \to A^1(U) \to A^1(W)$ is surjective and has kernel
\begin{equation}\label{kernelOfPullbackToWonderfulModelIsKernelOfPulllbackToBergmanFan}
	\ker\iota^\star = \ker( A^1(X) \to A^1(U))
\end{equation}
because $U \hookrightarrow X$ is an open immersion and because $A^\star(U) \to A^\star(W)$ is an isomorphism of graded rings by \cite[Theorem 3]{FY04}. Because $A^1(U)$ is generated by Chern classes of line bundles and $W \hookrightarrow X$ is a closed immersion of smooth varieties, we also have that $\iota^\star$ and $\iota_\star$ satisfy the projection formula
\[
	\iota_\star( \iota^\star\gamma \cap \beta) = \gamma \cap \iota_\star\beta
\]
for any $\gamma \in A^1(X)$ and $\beta \in A_1(W)$. Thus for any $\beta \in \ker\iota_\star$ and $\gamma \in A^1(X)$,
\[
	\deg( \iota^\star\gamma \cap \beta) = \deg( \gamma \cap \iota_\star\beta ) = 0,
\]
so $\beta = 0$ by the surjectivity of $\iota^\star$ and the fact that $\deg( \_ \,\cap\, \_ ): A^1(W) \times A_1(W) \to \ZZ$ is a perfect pairing. Thus $\iota_\star$ is injective. For any $\beta \in A_1(W)$ and $\gamma \in \ker( A^1(X) \to A^1(U))$,
\[
	\gamma \cap \iota_\star\beta = \iota_\star( \iota^\star\gamma \cap \beta) = 0,
\]
so $\iota_\star(A_1(W)) \subset A_1(M)$.

Now suppose $\beta \in A_1(M)$ and consider the map $\varphi: A^1(X) \to \ZZ: \gamma \mapsto \deg( \gamma \cap \beta)$. By the surjectivity of $\iota^\star$, (\ref{kernelOfPullbackToWonderfulModelIsKernelOfPulllbackToBergmanFan}), and the definition of $A_1(M)$, the map $\varphi$ factors as $A^1(X) \xrightarrow{\iota^\star} A^1(W) \xrightarrow{\psi} \ZZ$ for some $\psi$. Let $\beta' \in A_1(W)$ be such that $\psi: A^1(W) \to \ZZ$ is given by $\gamma' \mapsto \deg( \gamma' \cap \beta')$. Then for any $\gamma \in A^1(X)$,
\[
	\deg(\gamma \cap \iota_\star\beta') = \deg( \iota^\star \gamma \cap \beta') = \psi(\iota^\star(\gamma)) = \varphi(\gamma) = \deg(\gamma \cap \beta),
\]
so $\iota_\star\beta' = \beta$ by the fact that $\deg( \_ \,\cap\, \_ ): A^1(X) \times A_1(X) \to \ZZ$ is a perfect pairing. Thus $\iota_\star(A_1(W)) = A_1(M)$.
\end{proof}

The following corollary of Theorem \ref{VirtualFundamentalClassOfMatroidIsVirtualFundamentalClassOfArrangement} will show that, in a precise sense, the genus 0 Gromov--Witten theory of $W$ depends only on the matroid $M$.

Let $\eta: W \hookrightarrow U$ and $\rho: U \hookrightarrow X$ denote the inclusions, and for each $\gamma \in A^\star(U)$, fix some $\widetilde{\gamma} \in A^\star(X)$ such that $\rho^\star\widetilde{\gamma} = \gamma$. Note that the following determines all genus 0 Gromov--Witten invariants of $W$ because $\eta^\star: A^\star(U) \to A^\star(W)$ is an isomorphism by \cite[Theorem 3]{FY04}.

\begin{corollary}\label{CorollaryGromovWittenInvariantsAreCombinatorial}
Let $\gamma_1, \dots, \gamma_m \in A^\star(U)$. Then
\begin{align*}
	\cK_\Gamma(\rho \circ \eta)_\star&\left((\ev_1^\star(\eta^\star\gamma_1) \cup \dots \cup \ev_m^\star(\eta^\star\gamma_m) ) \cap [\cK_\Gamma(W)]^\mathsf{vir}\right)\\
	&= (\ev_1^\star(\widetilde{\gamma_1}) \cup \dots \cup \ev_m^\star(\widetilde{\gamma_m})) \cap [M]_\Gamma^\mathsf{vir},
\end{align*}
and
\[
	\int_{[\cK_\Gamma(W)]^\mathsf{vir}} \ev_1^\star(\eta^\star\gamma_1) \cup \dots \cup \ev_m^\star(\eta^\star\gamma_m) = \int_{[M]_\Gamma^\mathsf{vir}} \ev_1^\star(\widetilde{\gamma_1}) \cup \dots \cup \ev_m^\star(\widetilde{\gamma_m}).
\]
\end{corollary}

\begin{proof}
Because the $\widetilde{\gamma_i}$ and thus the $\ev_i^\star(\widetilde{\gamma_i})$ are generated by Chern classes of line bundles, we have a projection formula
\begin{align*}
	\cK_\Gamma(\rho \circ \eta)_\star&\left( \cK_\Gamma(\rho \circ \eta)^\star(\ev_1^\star(\widetilde{\gamma_1}) \cup \dots \cup \ev_m^\star(\widetilde{\gamma_m})) \cap [\cK_\Gamma(W)]^\mathsf{vir}\right)\\
	&= (\ev_1^\star(\widetilde{\gamma_1}) \cup \dots \cup \ev_m^\star(\widetilde{\gamma_m})) \cap \cK_\Gamma(\rho \circ \eta)_\star[\cK_\Gamma(W)]^\mathsf{vir}\\
	&= (\ev_1^\star(\widetilde{\gamma_1}) \cup \dots \cup \ev_m^\star(\widetilde{\gamma_m})) \cap [M]^\mathsf{vir}_\Gamma,
\end{align*}
where the second equality is due to Theorem \ref{VirtualFundamentalClassOfMatroidIsVirtualFundamentalClassOfArrangement}. The first equation of the theorem then follows from the fact that
\[
	\ev_i \circ \cK_\Gamma(\rho \circ \eta) = \rho \circ \eta \circ \ev_i
\]
for all $i \in \{1, \dots, m\}$. The second equation of the theorem is obtained by applying $\deg$ to both sides of the first equation.
\end{proof}

We may now prove Theorem \ref{QuantumCohomologyIsCombinatorialInvariant}.

\begin{proof}[Proof of Theorem \ref{QuantumCohomologyIsCombinatorialInvariant}]
Suppose $\cA_1$ and $\cA_2$ are hyperplane arrangements in $\PP^d$ realizing $M$. Then \cite[Theorem 3]{FY04} induces an isomorphism of graded vector spaces
\[
	QH^\star(W_{\cA_1}) \cong QH^\star(W_{\cA_2})
\] 
that is compatible with the quantum product by Proposition \ref{PropositionImageOfCurveClassesIsCombinatorial} and Corollary \ref{CorollaryGromovWittenInvariantsAreCombinatorial} applied to the $3$-pointed invariants giving the structure constants of the ring.
\end{proof}

\section{Toric logarithmic structure}

The main construction of subsection \ref{subsectionVFCOfMatroid} associates to each matroid $M$ virtual fundamental classes $[M]_{\Gamma}^\mathsf{vir}$. We recall that part of the discrete data $\Gamma$ includes a choice of boundary divisor $D_X\subset X(\Delta_n)$. In the present section, we specialize to the situation where $D_X$ is the full toric boundary of $X(\Delta_n)$, where the data of $[M]_{\Gamma}^\mathsf{vir}$ can be transformed into a purely combinatorial data structure, namely that of Minkowski weights~\cite{FS97}. Our approach here is inspired by an influential paper of Katz on toric and tropical intersection theory~\cite{Katz09}. 

Throughout the section, we fix a matroid $M$ of rank $d+1$ on $n+1$ elements. The associated Bergman fan $\Delta_M$ is a subfan $\Delta_n$ and is a union of cones of dimension $d$, see for instance~\cite{Kat16}. 

\subsection{Toric intersection theory} Let $Y$ be a complete toric variety with fan $\Sigma$. Let $c\in A^k(Y)$ be a Chow cohomology class. If $\sigma$ is a cone in $\Sigma$ of codimension $k$ and $V(\sigma)$ is the associated closed stratum of dimension $k$, the homology class $c\cap V(\sigma)$ has a well-defined degree. These degrees, ranging over all codimension $k$ cones, determine the Chow cohomology class~\cite{FS97}. Let $\Sigma^{(k)}$ be the cones in $\Sigma$ of codimension $k$. 

\begin{definition}
A $\QQ$-valued function $c$ on $\Sigma^{(k)}$ is \textit{balanced} if it satisfies the relation
\[
\sum_{\sigma\in \Sigma^{(k)}: \sigma\subset \tau} \langle u,n_{\sigma,\tau} \rangle \cdot c(\sigma) = 0,
\]
where $\tau$ is a cone of codimension $k+1$ and the vector $n_{\sigma,\tau}$ is the generator of the lattice of $\sigma$ relative to that of $\tau$. A balanced function of this form is a \textit{Minkowski weight} of codimension $k$. 
\end{definition}

The Minkowski weights of a fixed codimension form a group. The direct sum of the groups admits a ring structure, described in~\cite{FS97}. Fulton and Sturmfels prove the following theorem. 

\begin{theorem}[Fulton--Sturmfels]
The operational Chow cohomology ring of $Y$ is naturally isomorphic to the ring $\mathsf{MW}(\Sigma)$ of Minkowski weights on $\Sigma$.
\end{theorem}

If $\pi: X'\to X$ is an equivariant modification of proper toric varieties and $\Sigma'\to \Sigma$ is the induced subdivision of fans, there is an evident inclusion of Minkowski weight rings
\[
\mathsf{MW}(\Sigma)\hookrightarrow \mathsf{MW}(\Sigma')
\]
which coincides, under the above identification, with the pullback homomorphism $\pi^\star$. 

It is both traditional and helpful to view a Minkowski weight not as a function but as an object of polyhedral geometry. Specifically, as the union of the cones in $\Sigma$ on which $c$ is nonzero, decorated by the function values. These are precisely the \textit{balanced polyhedral complexes} of tropical geometry. In particular, the pullback homomorphism under a modification as above corresponds to a refinement of balanced fans with the obvious weighing on their cones.

\subsection{Toric embeddings of moduli} The virtual fundamental class $[M]_{\Gamma}^\mathsf{vir}$ is a homology class on the space $K_\Gamma(X)$, but as we have shown in the previous section, the class is equal to the virtual Poincare dual of the Euler class of a vector bundle. The conversion to Minkowski weights is facilitated by the following two results. The first arises from the geometry of Chow quotients, and describes the moduli space of pointed genus $0$ curves as a subvariety of a toric variety, transverse to its boundary~\cite{Kap93,Tev07}. 

\begin{theorem}
The moduli space $\Mbar_{0,n}$ is equal to the Chow quotient of the Grassmannian of linear planes $G(2,n)$ by the $n$-dimensional dilating torus $T$. The Pl\"ucker embedding 
\[
G(2,n)\hookrightarrow \PP^{\binom{n}{2}-1}
\]
is $T$ equivariant and determines a map of Chow quotients
\[
\Mbar_{0,n}\hookrightarrow Y'_{0,n}
\]
where $Y'_{0,n}$ is the toric Chow quotient of the Pl\"ucker projective space by $T$. There is a toric resolution of singularities $Y_{0,n}\to Y'_{0,n}$ which restricts to an isomorphism on $\Mbar_{0,n}$. Moreover, the toric stratification of $Y_{0,n}$ pulls back to the stratification of $\Mbar_{0,n}$ by the topological type of the universal curve. 
\end{theorem}

We fix a resolution  $Y_{0,n}$ guaranteed by the theorem, and thereby a smooth toric variety affording a strict closed embedding $\Mbar_{0,n}\hookrightarrow Y_{0,n}$. The next result concerns the space of logarithmic stable maps to a toric variety with respect to its full toric boundary, and supplies a simple description of this moduli space~\cite{R15b}\footnote{The methods originally used to prove this result involve non-archimedean geometry; a simpler and more direct logarithmic geometric proof may be found in~\cite{RW19}.}.

\begin{theorem}
The moduli space $\mathsf K_\Gamma(X)$ is realized as a toroidal modification
\[
\mathsf K_\Gamma(X)\to \Mbar_{0,n}\times X.
\]
In particular, it is logarithmically smooth and irreducible of the expected dimension. Its virtual fundamental class coincides with its fundamental class. 
\end{theorem} 

As a consequence the moduli space of stable maps also embeds in a toric variety.

\begin{corollary}
There is a sequence of morphisms
\[
\mathsf K_\Gamma(X)\hookrightarrow Y_\Gamma\to Y_{0,n}\times X
\]
where the first arrow is a closed embedding and the second arrow is a toric modification. Moreover, the toric stratification of $Y_\Gamma$ pulls back to the stratification of $\mathsf K_\Gamma(X)$ by tropical type. 
\end{corollary}

\subsection{virtual fundamental classes to Minkowski weights}\label{sec: virtual-mw} We remind the reader that $M$ is a matroid that is realizable over a characteristic $0$ field, and $\Gamma$ is a choice of logarithmic Gromov--Witten discrete data, fixing the degree, number of markings, their contact orders, and the boundary divisor on the permutohedral toric variety, which for this section has been fixed to be the full toric variety. The variety $W$ is a wonderful model for an arrangement complement realizing $M$.

We now convert the class $[\mathsf K_\Gamma(W)]^{\mathsf{vir}}$ into a balanced polyhedral complex using the toric embedding above. It will be helpful to introduce a logarithmic Chow ring first. See~\cite{Bar18,MR21}. 

\begin{definition}
Let $Z$ be a logarithmic scheme. The \textit{logarithmic Chow ring of $Z$} with respect to its logarithmic structure is
\[
\ell A^\star(Z):=\varinjlim A^\star(Z')
\]
where the limit is taken over all logarithmic modifications, i.e. proper, logarithmically \'etale, and birational morphisms $Z'\to Z$, and transition maps given by cohomological pullback. 
\end{definition}

A basic property of the logarithmic Chow ring is that it is insensitive to replacing $Z$ with a further logarithmic modification. The logarithmic Chow ring of a toric variety $X$ has an elegant description. Let $N_{\RR}$ denote the cocharacter space of the dense torus of $X$. A \textit{tropical cycle on $N_{\RR}$} is a Minkowski weight on some complete fan in $N_{\RR}$. As an immediate consequence of the naturality of the identification of Chow cohomology with Minkowski weights, the logarithmic Chow ring of a toric variety $X$, equipped with its standard logarithmic structure, is identified with the set of all tropical cycles on $N_{\RR}$. The group and ring structure have natural descriptions as well, but we refer the reader to the literature for details~\cite{FS97,KP08}. 

Consider a logarithmic modification
\[
\pi: \cK_\Gamma(X)'\to \cK_\Gamma(X)
\]
of moduli spaces. Since the morphism $\cK_\Gamma(W)\hookrightarrow \cK_\Gamma(X)$ is strict, the scheme theoretic pullback defines a Cartesian diagram of fine and saturated logarithmic schemes
\[
\begin{tikzcd}
\cK_\Gamma(W)'\arrow{d}\arrow{r}{j} & \cK_\Gamma(X)'\arrow{d}{\pi}\\
\cK_\Gamma(W)\arrow{r}{i} & \cK_\Gamma(X).
\end{tikzcd}
\]
The moduli space $\cK_\Gamma(W)'$ is also equipped with a virtual fundamental class, which pushes forward to the virtual fundamental class of $\cK_\Gamma(W)$, see~\cite[Section~3.5]{R19}. 

\begin{proposition}
The following equality of classes holds in the Chow homology group of $\cK_\Gamma(X)'$:
\[
j_\star[\cK_\Gamma(W)']^{\mathsf{vir}} = c_{\mathsf{top}}(\pi^\star \mathbb F)\cap [\cK_\Gamma(X)']. 
\]
\end{proposition}

\begin{proof}
The substack $\cK_\Gamma(W)'$ in $\cK_\Gamma(X)'$ is cut out by a section of the pullback bundle $\pi^\star \mathbb F$. The equality of classes now follows by the virtual compatibility argument in subsection~\ref{sectionLogQuantumLefschetz}, see Remark~\ref{remarkVirtualCompatibilityAndModifications}. 
\end{proof}

Let $\pi: \cK_\Gamma(X)'\to \cK_\Gamma(X)$ be a logarithmic modification with nonsingular source. Invoking the Poincare duality isomorphism, we can identify the pushforward of the virtual fundamental class of $\cK_\Gamma(W)'$ with the top Chern class of the vector bundle $\pi^\star\mathbb F$. Since the logarithmic Chow ring is insensitive to replacement of the space by a logarithmic modification, the proposition immediately implies that there is an element
\[
[\cK_\Gamma(W)]^{\mathsf{log}} \ \ \textrm{in} \ \ \ell A^\star(\cK_\Gamma(X)).
\]

Recall that we have constructed a sequence of strict closed embeddings
\[
\cK_\Gamma(W)\hookrightarrow \cK_\Gamma(X)\hookrightarrow Y_\Gamma.
\]

\begin{lemma}
Every logarithmic modification of $\cK_\Gamma(X)$ is pulled back from a toric modification of $Y_\Gamma$. 
\end{lemma}

\begin{proof}
Every logarithmic modification of $\cK_\Gamma(X)$ is induced by a subdivision of its toroidal fan. It is proved in~\cite{R15b} that the toroidal fan of $\cK_\Gamma(X)$ is a union of cones in the toric fan of $Y_\Gamma$, with the identification induced by the inclusion $\cK_\Gamma(X)\hookrightarrow Y_\Gamma$. The result follows. 
\end{proof}

The logarithmic virtual fundamental class gives rise to a logarithmic cohomology class on the toric variety $Y_\Gamma$. Recall that we have fixed a modification $\pi: \cK_\Gamma(X)'\to \cK_\Gamma(X)$. We may choose a modification of $Y_\Gamma$ that induces it, to produce a sequence of modifications:
\[
\begin{tikzcd}
\cK_\Gamma(W)' \arrow[hookrightarrow]{r}\arrow{d} & \cK_\Gamma(X)' \arrow[hookrightarrow]{r}\arrow{d} & Y'_\Gamma\arrow{d}\\
\cK_\Gamma(W) \arrow[hookrightarrow]{r}& \cK_\Gamma(X) \arrow[hookrightarrow]{r}& Y_\Gamma
\end{tikzcd}
\]
Note that both squares are cartesian. Moreover, if we insist that $Y'_\Gamma\to Y_\Gamma$ is a resolution of singularities, then by pushing forward the virtual fundamental class of $\cK_\Gamma(W)'$ to $Y'_\Gamma$ and applying Poincare duality, we obtain a Chow cohomology class $c'_\Gamma(M)$, and we view it in the logarithmic Chow ring $\ell A^\star(Y_\Gamma)$. 

\begin{proposition}
The element $c'_\Gamma(M)$ in $\ell A^\star(Y_\Gamma)$ is independent of the toric resolution $Y'_\Gamma\to Y_\Gamma$. 
\end{proposition}

\begin{proof}
Since any two birational modifications of $Y_\Gamma$ can be dominated by a resolution of singularities of their common refinement, it suffices to check that if $p: Y_\Gamma''\to Y_\Gamma'$ is a further toric modification by a smooth toric variety, then the class $c_\Gamma''(M)$ is equal to $c_\Gamma'(M)$ in $\ell A^\star(Y_\Gamma)$. By unwinding the definition of the direct limit, it suffices to show an equality
\[
p^\star c_\Gamma'(M) = c_\Gamma''(M) \ \ \textrm{in} \ \ A^\star(Y_\Gamma'').
\]
Since $\cK_\Gamma(X)\hookrightarrow Y_\Gamma$ is strict, we have the pullback square
\[
\begin{tikzcd}
\cK_\Gamma(X)'' \arrow[hookrightarrow]{r}\arrow{d} & Y''_\Gamma\arrow{d}\\
 \cK_\Gamma(X)' \arrow[hookrightarrow]{r}& Y'_\Gamma.
\end{tikzcd}
\]
The two vertical arrows have smooth source and target, and therefore are both local complete intersection morphisms of codimension $0$, and the corresponding Gysin pullbacks from the Chow groups of $\cK_\Gamma(X)'$ to those of $\cK_\Gamma(X)''$ coincide. The virtual fundamental classes of $\cK_\Gamma(W)''$ and $\cK_\Gamma(W)'$, after pushforward to the corresponding spaces of maps to $X$, are both equal to the top Chern class of the vector bundle $\mathbb F$, pulled back from $\cK_\Gamma(X)$. The claim above is now a consequence of compatibility of proper pushforward and Gysin pullback. The statement in the proposition follows. 
\end{proof}

\begin{definition}[The virtual weight]
The \textit{virtual Minkowski weight} of a matroid $M$ in discrete data $\Gamma$ as above, is the Minkowski weight determined by the logarithmic Chow class $c_\Gamma(M)$ in the logarithmic Chow ring $\ell A^\star(Y_\Gamma)$. 
\end{definition}

\begin{remark}
Note that the image of $[\mathsf K_\Gamma(W)']_\Gamma^{\mathsf{vir}} $ in $A_\star(\mathsf K_\Gamma(X)')$ is equal to the Poincar\'e dual of the pullback of the class $e_\Gamma(M)$ defined in subsection \ref{subsectionVFCOfMatroid}. In particular, $c_\Gamma(M)$ depends only on the discrete data $\Gamma$ and the matroid $M$ and even makes sense when $M$ is not realizable.
\end{remark}

\subsection{Reconstruction: Gromov--Witten theory from the virtual weights}\label{sec: reconstruction} We maintain the notation of the previous section. We show that they completely determine the logarithmic Gromov--Witten invariants. The discrete data $\Gamma$ determines, for each marked point $p_i$ a stratum $F_i\subset W$, given by the intersection of all divisors with respect to which $p_i$ has positive contact order. If $p_i$ has zero contact with all divisors, then $F_i$ is defined to be $W$ itself. We recall that the moduli space $\cK_\Gamma(W)$ admits evaluation morphisms and forgetful morphisms to $\Mbar_{0,n}$:
\[
\mathsf{ev}_i: \cK_\Gamma(W)\to F_i \ \ \ \epsilon: \cK_\Gamma(W)\to \Mbar_{0,n}. 
\]
Let $\psi_i$ denote the first Chern class of the cotangent line bundles associated to the $i$th marked point in $\Mbar_{0,n}$. Given cohomology classes $\gamma_i$ in $A^\star(F_i)$ and positive integers $k_1,\ldots, k_n$, we define the logarithmic Gromov--Witten invariants\footnote{We have defined the ancestor theory here. The descendent theory differs from this by boundary corrections, and may also be treated by these methods, but additional bookkeeping which we wish to avoid here.} by:
\[
\langle \tau_{k_1}(\gamma_1),\ldots,\tau_{k_n}(\gamma_n)\rangle := \int_{[\cK_\Gamma(W)]^{\mathsf{vir}}} \prod_{i=1}^n \psi_i^{k_i} \mathsf{ev}_i^\star \gamma_i. 
\]

Our main result in this section is to show that these invariants may all be computed from the virtual weight $c_\Gamma(M)$ in the cocharacter space of the fan of $Y_\Gamma$. Specifically, for each operator $\tau_{k_i}(\gamma_i)$ we will construct a Minkowski weight $t_{k_i}(\gamma_i)$ on the fan of the toric variety $Y_\Gamma$ and establish the following reconstruction theorem. 

\begin{theorem}\label{thm: unique-reconstruction}
The logarithmic Gromov--Witten invariant $\langle \tau_{k_1}(\gamma_1),\ldots,\tau_{k_n}(\gamma_n)\rangle$ is computed by the tropical intersection product of the Minkowski weight $c_\Gamma(M)$ with the Minkowski weights $t_{k_i}(\gamma_i)$. 
\end{theorem}

In other words, the balanced weighted fan $c_\Gamma(M)$ uniquely reconstructs the logarithmic Gromov--Witten theory of $W$. The remainder of this subsection is dedicated to proving this theorem. 

We recall several basic closed embeddings. The first is the inclusion of the wonderful model $W\hookrightarrow X$. We have also constructed a sequence of strict closed embeddings
\[
\cK_\Gamma(W)\hookrightarrow \cK_\Gamma(X)\hookrightarrow Y_\Gamma,
\]
as well as a strict closed embedding
\[
\Mbar_{0,n}\hookrightarrow Y_{0,n}. 
\]

We first note that the cotangent classes have precursors in the toric variety $Y_{0,n}$. 

\begin{lemma}
The Chow cohomology class $\psi_i$ is the pullback of a Chow cohomology class on $Y_{0,n}$ along the inclusion $\Mbar_{0,n}\hookrightarrow Y_{0,n}$.
\end{lemma}

\begin{proof}
This is proved in~\cite[Section~7]{Katz09}. 
\end{proof}

We next note that each constraint in $W$ has a precursor as well. 

\begin{lemma}
Let $\gamma$ be a class in $A^\star(W)$. There exists a cohomology class $\underline \gamma$ in $A^\star(X)$ whose image under the pullback map induced by $W\hookrightarrow X$ is $\gamma$. 
\end{lemma}

\begin{proof}
The Bergman fan of $W$ determines a union of cones in $\Delta_n$ and therefore an invariant open toric variety $U\hookrightarrow X$ and a factorization
\[
W\hookrightarrow U\hookrightarrow X
\]
into a closed embedding followed by an open embedding. The pullback morphism 
\[
A^\star(U)\to A^\star(W)
\]
is an isomorphism~\cite[Theorem~3]{FY04}. The pullback morphism
\[
A^\star(X) \to A^\star(U)
\]
is surjective by the excision sequence. The result follows. 
\end{proof}

In fact, we require a mild strengthening of the above lemma. Let $F\subset W$ be a stratum. The embedding $W\hookrightarrow X$ determines an embedding
\[
F\hookrightarrow E
\]
where $E$ is the smallest closed torus orbit containing $F$. Note that $E$ is a stratum in the toric variety associated to the permutohedron, and is therefore a product of smaller dimensional permutohedral toric varieties. Similarly, $F$ is a stratum in a wonderful model, and is therefore a product of smaller wonderful models~\cite[Section~4.3]{dCP95}. Moreover, the induced map
\[
F_1\times\cdots\times F_s\hookrightarrow E_1\times\cdots\times E_s
\]
is a product of embeddings of wonderful models. Since both the source and target are linear varieties in the sense of Totaro, they both satisfy the Chow--K\"unneth property~\cite[Section~3]{Tot14}. 

\begin{lemma}
Let $F\hookrightarrow E$ be as above and let $\gamma$ be a class in $A^\star(F)$. There exists a cohomology class $\underline \gamma$ in $A^\star(E)$ whose image under the pullback map induced by $F\hookrightarrow E$ is $\gamma$. 
\end{lemma}

\begin{proof}
Apply the Chow--K\"unneth property and reduce to the previous lemma. 
\end{proof}

There is a largest torus invariant open subscheme of $Y_\Gamma$ that contains $\cK_\Gamma(X)$, obtained as the union of all torus orbits intersecting $\cK_\Gamma(X)$. Denote this toric variety by $Y_\Gamma^\circ$. The following result controls the evaluation morphism.

\begin{lemma}
Let $\mathsf{ev}_i: \cK_\Gamma(X)\to E$ be an evaluation morphism. There is a factorization
\[
\begin{tikzcd}
\cK_\Gamma(X)\arrow{r}{\mathsf{ev}_i}\arrow[hookrightarrow]{d} & E \\
Y_\Gamma^\circ\arrow[swap]{ur}{\mathsf{e}_i} & 
\end{tikzcd}
\]
where $\mathsf e_i$ is an equivariant toric morphism. 
\end{lemma}

\begin{proof}
We recall that the \textit{intrinsic torus} of a very affine variety is the algebraic torus associated to the abelian group of global invertible functions up to constant functions~\cite[Section~3]{Tev07}. The dense torus in the toric variety $Y^\circ_\Gamma$ coincides with the intrinsic torus of the interior of $\cK_\Gamma(X)$. To see this, note that this interior is the product of $\cM_{0,n}$ with an algebraic torus, and the intrinsic torus of $\cM_{0,n}$ is the dense torus in $Y_{0,n}$ by~\cite[Section~6]{HKT}. Since the morphism $\mathsf{ev}_i$ is a dominant map of very affine varieties, we obtain the claimed factorization by~\cite[Section~3.1]{Tev07}.
\end{proof}

As a result, every evaluation class has a toric precursor, as we now record. 

\begin{lemma}
Let $E$ be a stratum of $X$ as above and fix a cohomology class $\gamma$ receiving an evaluation map from $\cK_\Gamma(X)$ along the marked point $p_i$. There is a cohomology class on $Y_\Gamma$ whose image under the pullback along
\[
\cK_\Gamma(X)\hookrightarrow Y_\Gamma
\]
is equal to the evaluation class $\mathsf{ev}_i^\star(\gamma)$. 
\end{lemma}

\begin{proof}
By applying the previous lemma, there is a cohomology class on $Y_\Gamma^\circ$ that pulls back to $\mathsf{ev}_i^\star(\gamma)$. The map $Y_\Gamma^\circ\hookrightarrow Y_\Gamma$ is an open embedding, so a precursor is guaranteed by the excision sequence. 
\end{proof}

\subsubsection{Proof of Theorem~\ref{thm: unique-reconstruction}} We begin with the Gromov--Witten invariant 
\[
\langle \tau_{k_1}(\gamma_1),\ldots,\tau_{k_n}(\gamma_n)\rangle := \int_{[\cK_\Gamma(W)]^{\mathsf{vir}}} \prod_{i=1}^n \psi_i^{k_i} \mathsf{ev}_i^\star \gamma_i. 
\]
By the projection formula, the invariant can be computed on the toric variety $Y_\Gamma$, provided that there exist classes in $A^\star(Y_\Gamma)$ that pull back to the integrand under the embedding 
\[
\cK_\Gamma(W)\hookrightarrow Y_\Gamma. 
\]
These are guaranteed by the preceding three lemmas, and we conclude.

\qed

\subsection{Support of the virtual weight} The class $c_\Gamma(M)$ is a Minkowski weight on the fan $Y_\Gamma$. We will denote the fan of $Y_\Gamma$ by $\Sigma_\Gamma$. Define its support $|c_\Gamma(M)|$ to be the set of cones in $\Sigma_\Gamma$ where the associated weight is nonzero. The support is tightly constrained, as we now explain. 

The moduli space $\cK_\Gamma(X)$ is logarithmically smooth, has an associated tropicalization which we denote $\mathsf T_\Gamma(\Delta_n)$. A point of the tropical moduli space parameterizes maps
\[
F: \plC\to \Delta_n
\]
where $\plC$ is a tropical rational curve and $F$ is an integer affine map satisfying the balancing condition. A detailed description is found in~\cite[Section~3.1]{R15b}. The subvariety $W\hookrightarrow X$ determines an open subscheme $U\hookrightarrow X$ consisting of the union of all locally closed torus orbits intersecting $W$. The equivariant inclusion $U\hookrightarrow X$ determines an inclusion of subcomplex $\Delta_M\hookrightarrow \Delta_n$, identifying the Bergman fan of $M$ with a subcomplex of the permutohedral fan. 

The tropical moduli space $\mathsf{T}_\Gamma(\Delta_n)$ contains a subset $\mathsf{T}_\Gamma(\Delta_M)$ of tropical maps whose image is contained in $\Delta_M$. 

\begin{lemma}
The subset $\mathsf{T}_\Gamma(\Delta_M)\subset \mathsf{T}_\Gamma(\Delta_n)$ is a union of cones. 
\end{lemma}

\begin{proof}
The cone complex $\Delta_M$ is a union of cones in $\Delta_n$. By~\cite[Section~3.2]{R15b}, the cones of $ \mathsf{T}_\Gamma(\Delta_n)$ are determined by the specification of cones of $\Delta_n$ to which each face of the source tropical curve maps. The condition that a moduli point of $\mathsf{T}_\Gamma(\Delta_n)$ lies in $\mathsf{T}_\Gamma(\Delta_M)$ is then precisely the condition that each such cone is contained in $\Delta_M$. This determines a subcomplex of $\mathsf{T}_\Gamma(\Delta_n)$ as claimed. 
\end{proof}

We can now state the support constraint on $c_\Gamma(M)$. 

\begin{proposition}
If $M$ is realized by a hyperplane arrangement over the complex numbers, then the support $|c_\Gamma(M)|$ is contained in $\mathsf{T}_\Gamma(\Delta_M)$.
\end{proposition}

\begin{proof}
Let $U\subset X$ be the open toric variety determined by the union of the cones of $\Delta_n$ that are contained in $\Delta_M$. Let $\cK_\Gamma(U)$ be the open substack of $\cK_\Gamma(X)$ parameterizing those logarithmic maps whose image is contained in $U$. If $W$ is the wonderful model of any realization of $M$, then $W\subset U$, and consequently $\cK_\Gamma(U)$ contains $\cK_\Gamma(W)$. We now compute the virtual weight $c_\Gamma(M)$ by the projection formula. Specifically, we must determine intersections of the boundary strata of the toric variety $Y_\Gamma$ with the virtual fundamental class $[\cK_\Gamma(W)]^{\mathsf{vir}}$. By the projection formula, the weight is determined as follows. We choose strata of $Y_\Gamma$ whose codimension is equal to the virtual dimension of $\cK_\Gamma(W)$. The pullback of this stratum is a stratum that may have positive dimension, but comes equipped with a virtual fundamental class in homological degree $0$. The degrees of these strata virtual fundamental classes, across all such strata of $Y_\Gamma$, determine the virtual Minkowski weight. 

However, as argued above, we have a factorization
\[
\cK_\Gamma(W)\hookrightarrow \cK_\Gamma(U)\hookrightarrow \cK_\Gamma(X)\hookrightarrow Y_\Gamma. 
\]
Moreover, the maps are all strict, and the stratification on $Y_\Gamma$ pulls back to the induced logarithmic stratification on each space mapping to it. It follows immediately that if we are given a stratum of $Y_\Gamma$ that does not meet $\cK_\Gamma(U)$ then the Minkowski weight on this cone must be zero. However, the strata of $Y_\Gamma$ that meet $\cK_\Gamma(U)$ are precisely those corresponding to cones in $\mathsf T_\Gamma(\Delta_M)$. The result follows.  
\end{proof}

\subsection{First examples}\label{sec: examples} The computation of Gromov--Witten invariants is a famously difficult problem, and the virtual Minkowski weights are similarly difficult to compute. In what follows, we sketch some simple examples that give a sense of the behaviour of the virtual Minkowski weights. 

The first example illustrates the relationship between our results here and traditional correspondence theorems~\cite{Gro14,Gro15,Mi03,NS06,R15b}.

\begin{example}
If $M$ is the matroid associated to an arrangement of $n+1$ generic hyperplanes in $\mathbb P^n$, then $W$ coincides with $X$, namely the toric variety of the permutohedron. In this case, the moduli space $\cK_\Gamma(W)$ is logarithmically unobstructed, and it is the closure of its interior. Moreover, this interior is isomorphic to $\cM_{0,n}$ times an algebraic torus~\cite{R19b}. In this case, the weight $c_\Gamma(M)$ is represented by the balanced polyhedral complex $T_\Gamma(\Delta_n)$ with weight $1$ on all maximal cones. 
\end{example}

We next show that the virtual fundamental class construction detects a nontrivial subcomplex of $\mathsf T_\Gamma(\Delta_M)$. 

\begin{example}
Let $M$ be the matroid associated to the arrangement of $3$ points on $\mathbb P^1$. Label the boundary divisor $3$ points as $q_1,q_2,q_3$ in $\mathbb P^1$. The Bergman fan $\Delta_M$ coincides with the $1$-skeleton of the fan of $\PP^2$, the ``traditional'' tropical line in the plane. We choose the discrete data $\Gamma$ to have $6$ marked points $p_{11},p_{12},p_{21},p_{22},p_{31},p_{32}$. We choose contact orders such that $p_{ij}$ has contact order $1$ with $q_i$, for all $i$ and $j$. 

The cone complex $T_\Gamma(\Delta_M)$ contains a $3$-dimensional cone, parameterizing maps
\[
\plC\to \Delta_M
\]
where $\plC$ is a stable $6$-pointed trivalent tropical curve with $3$ bounded edges. The map is uniquely specified by the balancing condition. However, the Minkowski weight $c_\Gamma(M)$ in this case is a $2$-dimensional balanced subcomplex of $T_\Gamma(\Delta_M)$. In fact, an explicit computation shows that $c_\Gamma(M)$ is simply the union of the $2$-dimensional cones in this case. 
\end{example}

The point of the previous example is that the combinatorial invariant $c_\Gamma(M)$ involves some subtlety. For example, the naive generalization of our first example would suggest that the virtual fundamental class is always equal to $T_\Gamma(\Delta_M)$ with weights equal to $1$ which is obviously a combinatorial invariant. The above example shows that this cannot be the case, indeed, the complex $T_\Gamma(\Delta_M)$ is not even equidimensional!

In fact, if the matroid is allowed to be of higher rank the phenomenon is already visible in degree $1$. 

\begin{example}
This example is adapted from work of Lamboglia~\cite{Lam18}. In this work, it is shown that if $M$ is the matroid associated to a generic $2$-dimensional plane in $\mathbb P^n$ for $n\geq 5$, the tropical Fano variety is not pure dimensional~\cite[Section~3]{Lam18}. Note that $\Delta_M$ is the $2$-skeleton of the fan of $\mathbb P^n$. The tropical Fano variety coincides with the space $T_\Gamma(\Delta_M)$ with the discrete data $\Gamma$ being that of a generic line. That is, there are $n+1$ marked points, with each meeting precisely one of the $n+1$ lines determined by the intersection of $\mathbb P^2$ with the boundary of $\mathbb P^n$. The weight $c_\Gamma(M)$ is supported on the $2$-dimensional cones of $T_\Gamma(\Delta_M)$. 
\end{example}

Additional higher dimensional examples can be computed by means of products of rank $2$ matroids, using the product formula in~\cite{R19b}. 

\begin{example}
Let $M'$ be any complex realizable matroid, and let $M$ be the direct sum of this matroid with the arrangement of $2$ points on $\mathbb P^1$. In this case, $\Delta_M$ is identified with $\Delta_{M'}\times\RR$. Let $W'$ be a wonderful realization of $M'$. We may take $W$ to be $W'\times \PP^1$. Then for any contact orders $\Gamma$, we obtain a morphism of moduli spaces
\[
\cK_\Gamma(W)\to \cK_\Gamma(W').
\]
The relative logarithmic tangent bundle of $W$ over $W'$ is trivial, and since we work in genus $0$, this map is logarithmically smooth. Since the logarithmic obstruction theory is compatible with logarithmically smooth morphism, it follows that the virtual Minkowski weight $c_\Gamma(M)$ is canonically identified with the preimage of $c_\Gamma(M')$ under the linear projection
\[
T_\Gamma(\Delta_{M})\to T_\Gamma(\Delta_M'). 
\]
Note that when $\Delta_{M'}$ in the previous example is a point, we recover the results of~\cite{CMR14b,R15b}. 
\end{example}

There is an in-principle way to determine the weight $c_\Gamma(M)$. We record it for use in future work. 

\begin{remark}\label{rem: computation}
If $M$ is realizable, the weight $c_\Gamma(M)$ may be computed as follows. For each cone of $T_\Gamma(\Delta_M)$ with dimension equal to the virtual dimension, we obtain a closed substack of $\cK_\Gamma(W)$. The substack possesses a virtual fundamental class in homological dimension $0$ by construction. Its degree is a rational number which is the value of $c_\Gamma(\Delta_M)$ on this cone. The number can in principle be calculated by the degeneration formula~\cite{R19b} or by Grothendieck--Riemann--Roch, which reduces the calculation to a tautological integral on the space of maps to a toric variety. 
\end{remark}

\section{Questions}\label{sec: questions}

We collect a number of questions suggesting future development of the directions here. 

In order to define Gromov--Witten invariants, it is necessary to intersect the matroid virtual class, which is a class in the Chow group of $\mathsf K_\Gamma(X(\Delta_n))$, with cohomology classes in the Chow ring of $M$. We recall that this Chow ring is identified with the Chow ring $A^\star(U_M)$ for $U_M\hookrightarrow X(\Delta_n)$ the torus invariant open determined by the Bergman fan of $M$ inside $\Delta_n$. Since we only have access to pulling back cohomology classes from $X(\Delta_n)$ and its strata, we must first choose a lift of the relevant classes along the surjection
\[
A^\star(X(\Delta_n))\to A^\star(U_M)
\]
induced by flat pullback along $U_M\hookrightarrow X(\Delta_n)$. The following question asks whether the resulting invariants are well-defined, i.e. if they are independent of the lift. 

\begin{questions}
Let $M$ be a matroid on $\{0, \dots, n\}$, and let $\gamma_1, \dots, \gamma_m \in A^\star(X(\Delta_n))$ be such that $\gamma_i{|_{U_M}} = 0$ for some $i \in \{1, \ldots, m\}$. Do we have
\[
\mathsf{deg}	\left(\mathsf{ev}_1^\star(\gamma_1) \cup \dots \cup \mathsf{ev}_m^\star(\gamma_m)\right) \cap [M]^{\mathsf{vir}}_\Gamma = 0
\]
for all choices of numerical data $\Gamma$?
\end{questions}

A positive answer to this question would provide a definition of Gromov--Witten invariants of an arbitrary matroid $M$ in genus $0$, and produce a quantum deformation of the Chow ring of $M$ studied in~\cite{FY04}. 

In the penultimate section, we constructed Minkowski weights $c_\Gamma(M)$ associated to any matroid. In the realizable case, we were able to show that the support of $c_\Gamma(M)$ was contained in the cone complex $T_\Gamma(\Delta_M)$. 

\begin{questions}
Is the support of the virtual Minkowski weight $c_\Gamma(M)$ contained in $T_\Gamma(\Delta_M)$?
\end{questions}

At present, we have few tools to calculate the virtual Minkowski weights associated to a matroid. Invariants of matroids, such as the Tutte polynomial and characteristic polynomial, often satisfy good properties with respect to deletion and contraction operations. 

\begin{questions}
What is the behavior of the virtual Minkowski weights $c_\Gamma(M)$ under deletion/contraction of $M$ at an element of the ground set?
\end{questions}

We expect that the degeneration formula may shed some light on this question~\cite{R19}. 

Another basic question is to follow the ideas of this paper, but in higher genus. 

\begin{questions}
Is the genus $g$ logarithmic Gromov--Witten theory of the wonderful model of a complex arrangement complement a combinatorial invariant? If so, can a higher genus Gromov--Witten theory for a general matroid be defined?
\end{questions}

The obstacle here is that virtual fundamental classes in higher genus are not easily expressible via Chern class operations on the space of maps to $X(\Delta_n)$. However, in genus $1$, the reduced logarithmic Gromov--Witten theory developed in~\cite{RSW17B} may provide an avenue of access.

\bibliographystyle{siam} 
\bibliography{MatroidGW}

\end{document}